\documentclass[a4paper]{article}

\usepackage{amsmath,amsthm}      % For mathematical symbols
\usepackage{amssymb}      % For mathematical symbols
\usepackage{enumerate}    % To support \begin{enumerate}[(i)]
\usepackage{graphicx}     % For embedding figures

\newcommand{\Hom}{\mathbf{H}}
\newcommand{\Cycles}{\mathbf{Z}}
\newcommand{\Boundaries}{\mathbf{B}}
\newcommand{\Chains}{\mathbf{C}}
\newcommand{\C}{\mathbb{C}}
\newcommand{\co}{\colon\thinspace}
\newcommand{\curr}{E^\ast}
\newcommand{\desc}{\prec}
\newcommand{\dotcup}{\mathbin{\dot{\cup}}}

\newcommand{\parent}[1]{\hat{#1}}

\newcommand{\tri}{\mathfrak{T}}
\newcommand{\tv}{\mathrm{TV}}
\newcommand{\tw}{\mathrm{tw}}
\newcommand{\Z}{\mathbb{Z}}
\newcommand{\zmin}{z^-}
\newcommand{\zmax}{z^+}
\newcommand{\Q}{\mathbb{Q}}

\newtheorem{theorem}{Theorem}
\newtheorem{lemma}[theorem]{Lemma}
\newtheorem{corollary}[theorem]{Corollary}
\newtheorem{proposition}[theorem]{Proposition}
\newtheorem{algorithm}[theorem]{Algorithm}

\theoremstyle{definition}
\newtheorem*{definition}{Definition}
\newtheorem*{notation}{Notation}

\begin{document}
%
%\pagestyle{headings}  % switches on printing of running heads

%%%%%%%%%%%%%%%%%%%%%%%%%%%%%%%%%%%%%%%%%%%%%%%%%%%%%%%%%%%%%%%%%%%%%%%%
%
%   Title and abstract
%
%%%%%%%%%%%%%%%%%%%%%%%%%%%%%%%%%%%%%%%%%%%%%%%%%%%%%%%%%%%%%%%%%%%%%%%%

\title{Algorithms and complexity for \\ Turaev-Viro invariants}
\author{Benjamin A.\ Burton \and Cl{\'e}ment Maria \and
Jonathan Spreer\thanks{The University of Queensland, Brisbane QLD 4072, Australia, \texttt{bab@maths.uq.edu.au}, \texttt{c.maria@uq.edu.au}, \texttt{j.spreer@uq.edu.au}}}

\date{}
\maketitle

\begin{abstract}
The Turaev-Viro invariants are a powerful family of topological invariants
for distinguishing between different 3-manifolds.  They are
invaluable for mathematical software, but current algorithms to compute
them require exponential time.

The invariants are parameterised by an integer $r \geq 3$.
We resolve the question of complexity for $r=3$ and $r=4$, giving simple
proofs that computing Turaev-Viro invariants for $r=3$ is polynomial
time, but for $r=4$ is \#P-hard.  Moreover, we give an explicit
fixed-parameter tractable algorithm for arbitrary $r$, and show through
concrete implementation and experimentation that this algorithm is
practical---and indeed preferable---to the prior state of the art
for real computation.
\end{abstract}


\section{Introduction}

In geometric topology, testing homeomorphism (topological equivalence)
is a fundamental algorithmic problem. However, beyond dimension two
it is remarkably difficult.
In dimension three---the focus of this paper---an algorithm follows
from Perelman's proof of
the geometrisation conjecture \cite{kleiner08-perelman}, but
it is extremely intricate, its complexity is unknown and it has never 
been implemented.
%In dimensions $\geq 4$, the problem becomes undecidable
%\cite{markov60-insolubility}.

As a result, practitioners in computational topology rely on
simpler \emph{topological invariants}---computable properties
of a topological space that can be used to tell different spaces apart.
One of the best known invariants is homology,
%which can be computed using
%polynomial time matrix reduction algorithms.  However,
but for 3-manifolds (the
3-dimensional generalisation of surfaces) this is weak:
there are many topologically different 3-manifolds that homology
cannot distinguish.
Therefore major software packages in 3-manifold topology rely on
invariants that are stronger but more difficult to compute.

In the discrete setting, among the most useful invariants for
3-manifolds are the \emph{Turaev-Viro invariants} \cite{turaev92-invariants}.
These are analogous to the Jones polynomial for knots:
they derive from quantum field theory, but offer a much simpler
combinatorial interpretation that lends itself well to algorithms and
exact computation.
They are implemented in the major software packages \emph{Regina}
\cite{regina} and the \emph{Manifold Recogniser}
\cite{matveev03-algms,recogniser}, and they play a key role in
developing census databases, which are analogous to the well-known
dictionaries of knots \cite{burton07-nor7,matveev03-algms}.
Their main difficulty is that they are
slow to compute: current implementations \cite{regina,recogniser} are
based on backtracking searches, and require exponential time.

The purpose of this paper is threefold: (i)~to introduce the Turaev-Viro
invariants to the wider computational topology community;
(ii)~to understand the complexity of computing these invariants; and
(iii)~to develop new algorithms that are suitable for practical software.

The Turaev-Viro invariants
are parameterised by two integers $r$ and $q$, with $r \geq 3$;
we denote these invariants by $\tv_{r,q}$.
A typical algorithm for computing $\tv_{r,q}$ will take as input a
triangulated 3-manifold,
composed of $n$ tetrahedra attached along their triangular faces;
we use $n$ to indicate the input size.
For all known algorithms,
the difficulty of computing $\tv_{r,q}$ grows significantly as $r$
increases (but in contrast, the difficulty is essentially independent of $q$).

Our main results are as follows.

\begin{itemize}
\item
Kauffman and Lins \cite{kauffman91-tv} state that for $r=3,4$
one can compute $\tv_{r,q}$ via ``simple and efficient methods
of linear algebra'', but they give no details on either the
algorithms or the complexity.
We show here that in fact the situations for $r=3$ and $r=4$
are markedly different:
computing $\tv_{r,q}$ for orientable manifolds
and $r=3$ is polynomial time, but for $r=4$
is \#P-hard.

\item
We give an explicit algorithm for computing $\tv_{r,q}$ for general $r$
that is fixed-parameter tractable (FPT).  Specifically, for any fixed $r$
and any class of input triangulations whose dual graphs have bounded
treewidth, the algorithm has running time linear in $n$.
Furthermore, we show through comprehensive experimentation that this
algorithm is \emph{practical}---we implement it in the open-source
software package \emph{Regina} \cite{regina}, run it through exhaustive
census databases, and find that this new FPT algorithm is
comparable to---and often significantly faster than---the
prior backtracking algorithm.

\item
We give a new geometric interpretation of the formula for $\tv_{r,q}$,
based on systems of ``normal arcs'' in triangles.  This generalises
earlier observations of Kauffman and Lins for $r=3$ based on embedded
surfaces \cite{kauffman91-tv}, and offers an interesting potential
for future algorithms based on Hilbert bases.
\end{itemize}

The \#P-hardness result for $r=4$ is the first classical hardness
result for the Turaev-Viro invariants.\footnote{%
    For \emph{quantum} computation, approximating
    Turaev-Viro invariants is universal \cite{alagic10-tv}.}
However, the proofs for this and the polynomial-time $r=3$ result
are simple: the algorithm for $r=3$ derives from a known homological
formulation \cite{matveev03-algms}, and the result for $r=4$ adapts
Kirby and Melvin's NP-hardness proof for the
more complex Witten-Reshetikhin-Turaev invariants \cite{kirby91-witten}.

The FPT algorithm for general $r$ is significant in that it is not just
theoretical, but also practical---and indeed \emph{preferable}---for real
software.  It was previously known that computing $\tv_{r,q}$ is
FPT \cite{burton14-courcelle}, but that
prior result was purely existential, and
would lead to infeasibly large constants in the running time if
translated to a concrete algorithm.  More generally, FPT algorithms do not
always translate well into practical software tools, and this paper is
significant in giving the first demonstrably practical FPT algorithm
in 3-manifold topology.

%Throughout this paper, we use a model of computation in which
%basic arithmetical operations on the rationals are constant time.
%Since our running times are dominated by combinatorial searches
%(and not the rational arithmetic), this assumption is relatively benign.

\section{Preliminaries} \label{s-prelim}

Let $M$ be a closed 3-manifold.
A \emph{generalised triangulation} of $M$ is a collection of
$n$ abstract tetrahedra $\Delta_1,\ldots,\Delta_n$ equipped with
affine maps that identify (or ``glue together'') their
$4n$ triangular faces in pairs,
so that the underlying topological space is homeomorphic to $M$.

In particular, as a consequence of the face identifications,
it is possible that several vertices of the same tetrahedron may be
identified together (and likewise for edges and triangles).
Indeed, it is common in practical applications to
have a \emph{one-vertex triangulation}, in which all vertices of all
tetrahedra are identified to a common point.
In general, the $4n$ tetrahedron vertices are partitioned into
equivalence classes according to how they are identified together;
we refer to each such equivalence class as a single
\emph{vertex of the triangulation}, and likewise for edges and triangles.

Generalised triangulations are widely used across major 3-manifold
software packages.  They are (as the name suggests)
more general than simplicial complexes, which allows them to
express a rich variety of different 3-manifolds using very few tetrahedra.
For instance, with just $n \leq 11$ tetrahedra one can create
%50\,817 distinct minimal triangulations of
13\,400 distinct prime orientable 3-manifolds
\cite{burton11-genus,matveev03-algms}.

\subsection{The Turaev-Viro invariants}
\label{ssec:tv}

%Here we focus our attention only on the specific parameterised family of
%invariants presented by Turaev and Viro in their original paper
%\cite{turaev92-invariants}.
%In their paper they also present a broader
%framework for constructing additional invariants (now called
%\emph{invariants of Turaev-Viro type}), but we do not discuss these here.
%
Let $\tri$ be a generalised triangulation of a closed 3-manifold $M$,
and let $r$ and $q$ be integers with $r \geq 3$, $0 < q < 2r$, and $\gcd(r,q)=1$.
We define the Turaev-Viro invariant $\tv_{r,q}(\tri)$ as follows.

Let $V$, $E$, $F$ and $T$ denote the set of vertices, edges,
triangles and tetrahedra respectively of the triangulation $\tri$.
Let $I = \{0, 1/2, 1, 3/2, \ldots, (r-2)/2\}$; note that $|I|=r-1$.
We define a \emph{colouring} of $\tri$ to be a map $\theta\co E \to I$;
that is, $\theta$ ``colours'' each edge of $\tri$ with an element
of $I$.  A colouring $\theta$ is
\emph{admissible} if, for each triangle of $\tri$, the three edges
$e_1$, $e_2$, and $e_3$ bounding the triangle satisfy:
\begin{itemize}
    \item the \emph{parity condition} $\theta(e_1)+\theta(e_2)+\theta(e_3)\in \Z$;
    \item the \emph{triangle inequalities}
        $\theta(e_1) \leq \theta(e_2) + \theta(e_3)$,
        $\theta(e_2) \leq \theta(e_1) + \theta(e_3)$, and
        $\theta(e_3) \leq \theta(e_1) + \theta(e_2)$; and
    \item the \emph{upper bound constraint}
        $\theta(e_1)+\theta(e_2)+\theta(e_3)\leq r-2$.
\end{itemize}
More generally, we refer to any triple $(i,j,k) \in I \times I \times I$
satisfying these three conditions as an \emph{admissible triple} of colours.

For each admissible colouring $\theta$ and
for each vertex $v \in V$, edge $e \in E$, triangle $f \in F$
or tetrahedron $t \in T$, we define \emph{weights}
$|v|_\theta, |e|_\theta, |f|_\theta, |t|_\theta \in \C$.

\bigskip
Our notation differs slightly from Turaev and Viro \cite{turaev92-invariants};
most notably, Turaev and Viro do not consider
triangle weights $|f|_\theta$, but instead incorporate an additional
factor of $|f|_\theta^{1/2}$ into each tetrahedron weight
$|t|_\theta$ and $|t'|_\theta$ for the two tetrahedra $t$ and $t'$ containing $f$.
This choice of notation simplifies the notation and avoids unnecessary
(but harmless) ambiguities when taking square roots.

Let $\zeta = e^{i \pi q / r} \in \C$.  Note that our conditions imply that
$\zeta$ is a $(2r)$th root of unity, and that $\zeta^2$ is a
\emph{primitive} $r$th root of unity; that is,
$(\zeta^2)^k \neq 1$ for $k=1,\ldots,r-1$.
For each positive integer $i$, we define
$[i] = (\zeta^i-\zeta^{-i})/(\zeta-\zeta^{-1})$ and,
as a special case, $[0] = 1$.
We next define the ``bracket factorial''
$[i]! = [i]\,[i-1] \ldots [0]$.
Note that $[r] = 0$, and thus $[i]! = 0$ for all $i \geq r$.

We give every vertex constant weight
\begin{equation*}
|v|_\theta = \frac{\left|\zeta-\zeta^{-1}\right|^2}{2r} ,
\end{equation*}
and to each edge $e$ of colour $i \in I$ (i.e.,
for which $\theta(e) = i$) we give the weight
\begin{equation*}
|e|_\theta = (-1)^{2i} \cdot [2i+1].
\end{equation*}
A triangle $f$ whose three edges have colours $i,j,k \in I$ is assigned the
weight
\[ |f|_\theta = (-1)^{i+j+k} \cdot
    \frac{[i+j-k]! \cdot [i+k-j]! \cdot [j+k-i]!}{[i+j+k+1]!}. \]
Note that the parity condition and triangle inequalities ensure that
the argument inside each bracket factorial is a non-negative integer.

\begin{figure}
    \begin{center}
        \includegraphics[width = .15\textwidth]{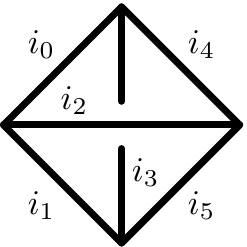}
        \caption{Edge colours of a tetrahedron. \label{fig:tet}}
    \end{center}
\end{figure}

Finally, let $t$ be a tetrahedron with edge colours
$i_0,i_1,i_2,i_3,i_4,i_5$ as indicated in Figure~\ref{fig:tet}. In particular,
the four triangles surrounding $t$ have colours
$(i_0,i_1,i_3)$, $(i_0,i_2,i_4)$, $(i_1,i_2,i_5)$ and $(i_3,i_4,i_5)$,
and the three pairs of opposite edges have colours
$(i_0,i_5)$, $(i_1,i_4)$ and $(i_2,i_3)$.  We define
\begin{align*}
\tau_\phi(t,z) &=
    [z-i_0-i_1-i_3]! \cdot [z-i_0-i_2-i_4]! \cdot
    [z-i_1-i_2-i_5]! \cdot [z-i_3-i_4-i_5]!\,, \\
\kappa_\phi(t,z) &= [i_0+i_1+i_4+i_5-z]! \cdot
                  [i_0+i_2+i_3+i_5-z]! \cdot
                  [i_1+i_2+i_3+i_4-z]!
\end{align*}
for all integers $z$ such that the bracket factorials above all have non-negative
arguments; equivalently, for all integers $z$ in the range
$\zmin \leq z \leq \zmax$ with
\begin{align*}
\zmin &= \max\{i_0+i_1+i_3,\ i_0+i_2+i_4,\ i_1+i_2+i_5,\ i_3+i_4+i_5\}\,; \\
\zmax &= \min\{i_0+i_1+i_4+i_5,\ i_0+i_2+i_3+i_5,\ i_1+i_2+i_3+i_4\}.
\end{align*}
Note that, as before, the parity condition ensures
that the argument inside each bracket factorial above is an integer.
We then declare the weight of tetrahedron $t$ to be
\begin{equation*}
|t|_\phi = \sum_{\zmin \leq z \leq \zmax}
    \frac{(-1)^z \cdot [z+1]!}{\tau_\phi(t,z) \cdot \kappa_\phi(t,z)},
\end{equation*}

\bigskip
Note that all weights are polynomials on $\zeta$ with rational coefficients, 
where $\zeta = e^{i \pi q/r}$.

Using these weights, we define the \emph{weight of the colouring} to be
\begin{equation}
|\tri|_\theta =
    \prod_{v \in V} |v|_\theta \times
    \prod_{e \in E} |e|_\theta \times
    \prod_{f \in F} |f|_\theta \times
    \prod_{t \in T} |t|_\theta,
\label{eqn-cweight}
\end{equation}
and the Turaev-Viro invariant to be the sum
over all admissible colourings
\[ \tv_{r,q}(\tri) = \sum_{\theta\ \mathrm{admissible}} |\tri|_\theta. \]

In \cite{turaev92-invariants}, Turaev and Viro
show that $\tv_{r,q}(\tri)$ is indeed
an invariant of the manifold; that is, 
if $\tri$ and $\tri'$ are generalised triangulations of the same
closed 3-manifold $M$, then $\tv_{r,q}(\tri) = \tv_{r,q}(\tri')$
for all $r,q$.
Although $\tv_{r,q}(\tri)$ is defined on the complex numbers $\C$,
it always takes a real value
(more precisely, it is the square of the modulus of a
Witten-Reshetikhin-Turaev invariant) \cite{walker91-notes}.

\subsection{Treewidth and parameterised complexity}

%We begin this section with some terminology:
Throughout this paper we always refer to \emph{nodes} and \emph{arcs}
of graphs, to clearly distinguish these from the
\emph{vertices} and \emph{edges} of triangulations.

Robertson and Seymour introduced the concept of the \emph{treewidth} of a
graph \cite{robertson86-algorithmic}, which now plays
a major role in parameterised complexity. Here, we adapt this concept to
triangulations in a straightforward way.

\begin{definition}
Let $\tri$ be a generalised triangulation of a 3-manifold, and let
$T$ be the set of tetrahedra in $\tri$.
A \emph{tree decomposition} $(X, \{B_\tau\})$ of $\tri$ consists of
a tree $X$ and
{\em bags} $B_\tau \subseteq T$ for each node $\tau$ of $X$, for which:
\begin{itemize}
    \item each tetrahedron $t \in T$ belongs to some bag $B_\tau$;
    \item if a face of some tetrahedron $t_1 \in T$ is identified with
    a face of some other tetrahedron $t_2 \in T$,
    then there exists a bag $B_\tau$ with $t_1,t_2 \in B_\tau$;
    \item for each tetrahedron $t \in T$,
    the bags containing $t$ correspond to a connected subtree of $X$.
\end{itemize}
The \emph{width} of this tree decomposition is defined as $\max |B_\tau|-1$.
The \emph{treewidth of $\tri$}, denoted $\tw(\tri)$,
is the smallest width of any tree decomposition of $\tri$.
\end{definition}

The relationship between this definition and the
classical graph-theoretical notion of
treewidth is simple: $\tw(\tri)$ is the treewidth of the
\emph{dual graph} of $\tri$, the $4$-valent multigraph whose nodes
correspond to tetrahedra of $\tri$ and whose arcs represent
pairs of tetrahedron faces that are identified together.

Figure~\ref{fig:example} shows the dual graph of a $9$-tetrahedra 
triangulation of a $3$-manifold, along with a possible tree decomposition.
The largest bags have size three, and so the width of this tree
decomposition is $3-1=2$.
%In fact this is the smallest width possible, and so
%in this case $\tw(\tri)=2$.
%In addition, the graph on the left is not a tree and hence
%any tree decomposition must have bags of size at least three. Altogether,
%it follows that the treewidth of the graph is exactly $2$.

\begin{figure}[tb]
\centering
\includegraphics[width=.9\textwidth]{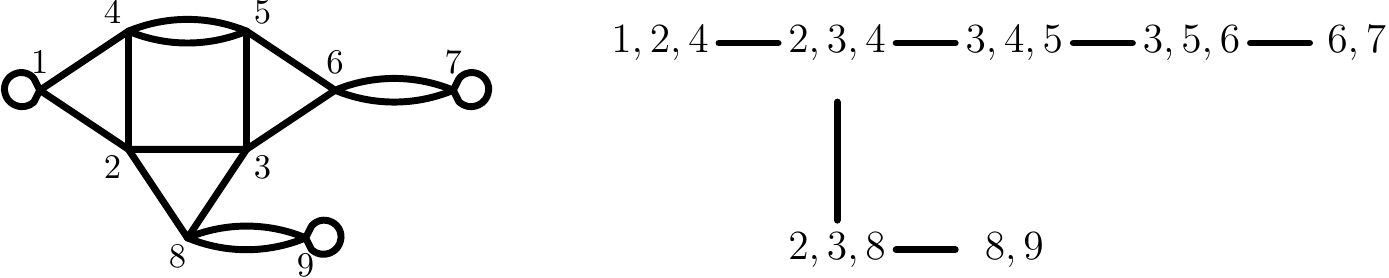}
\caption{The dual graph and a tree decomposition
    of a $3$-manifold triangulation}
\label{fig:example}
\end{figure}

%The structure of a tree decomposition is important when formulating  
%FPT-algorithms. Thus, the
%following Definition~\ref{def:nice_treewidth} in conjunction with
%Lemma~\ref{lem:nice_treewidth} are of great practical use in the 
%field of parameterised complexity.

\begin{definition}
\label{def:nice_treewidth} 
A \emph{nice tree decomposition} of a generalised triangulation $\tri$
is a tree decomposition $(X, \{B_\tau\})$ of $\tri$ whose
underlying tree $X$ is rooted, and where:
%(we suppose the decomposition tree T to be rooted at some arbitrary but fixed bag):
	\begin{itemize}
	\item The bag $B_\rho$ at the root of the tree is empty
    ($B_{\rho}$ is called the {\em root bag});
	\item If a bag $B_{\tau}$ has no children, then $|B_{\tau}|=1$ 
    (such a $B_{\tau}$ is called a {\em leaf bag});
	\item If a bag $B_{\tau}$ has two children $B_{\sigma}$ and $B_{\mu}$, then 
    $B_{\tau} = B_{\sigma} = B_{\mu}$ (such a $B_{\tau}$ is called a 
    {\em join bag});
	\item Every other bag $B_{\tau}$ has precisely one child $B_{\sigma}$,
    and either:
		\begin{itemize}
			\item $|B_{\tau}| = |B_{\sigma}| + 1$ and $B_{\tau} \supset B_{\sigma}$ 
        (such a $B_{\tau}$ is called an {\em introduce bag}), or
			\item $|B_{\tau}| = |B_{\sigma}| - 1$ and $B_{\tau} \subset B_{\sigma}$ 
        (such a $B_{\tau}$ is called a {\em forget bag}).
		\end{itemize}
	\end{itemize}
\end{definition}

Given a tree decomposition of a triangulation $\tri$ of width 
$k$ and $O(n)$ bags,
we can convert this in $O(n)$ time
into a \emph{nice} tree decomposition of $\tri$ that also has width $w$ and
$O(n)$ bags \cite{kloks1994treewidth}.

\section{Algorithms for computing Turaev-Viro invariants} \label{s-alg}

All of the algorithms in this paper use exact arithmetic.
This is crucial if we wish to avoid floating-point numerical instability,
since computing $\tv_{r,q}$ may involve exponentially many
arithmetic operations.

We briefly describe how this exact arithmetic works.
Since all weights in the definition of $\tv_{r,q}$ are
rational polynomials in $\zeta = e^{i \pi q/r}$,
all arithmetic operations remain within the rational
field extension $\Q(\zeta)$. If $\zeta$ is a primitive $n$th root of unity
then this field extension is called the \emph{$n$th cyclotomic field}.
This in turn is isomorphic to the polynomial field 
$\Q[X] / \Phi_n(X)$, where $\Phi_n(X)$ is the $n$th cyclotomic
polynomial with degree $\varphi(n)$ (Euler's totient function).
Therefore we can implement exact arithmetic using degree $\varphi(n)$
polynomials over $\Q$.

If $r$ is odd and $q$ is even, then $\zeta$ is a primitive
$r$th root of unity, and $\Q(\zeta) \cong \Q[X] / \Phi_r(X)$.
Otherwise $\zeta$ is a primitive $(2r)$th root of unity,
and $\Q(\zeta) \cong \Q[X] / \Phi_{2r}(X)$.
In this paper we give our complexity results in terms of
arithmetic operations in
$\Q(\zeta)$.

\bigskip
Let $\zeta$ be an $n$th root of unity and $\Q(\zeta)$ be the $n$th cyclotomic field. 
We represent elements of $\Q(\zeta)$ by polynomials of degree at most $\varphi(n)$, 
with rational coefficients, using the isomorphism $\Q(\zeta) \cong \Q[X] / \Phi_n(X)$. 
Asymptotically, the Euler totient function satisfies $\varphi(n) = \Theta(n)$. 
% In both cases, two distinct polynomials always represent two distinct complex numbers.
% Depending on the parity of $r$ and $q$, we represent colours by polynomials in 
% the respective polynomial field described above, and refer to it with $\Q(\zeta)$ in the following.
% In both cases, these polynomial have degree less than 
% $\varphi(2r) = O(r)$. 
Additions of two polynomials of degree at most $n$ 
are performed in $O(n)$ operations in $\Q$, and multiplications and divisions are 
performed in $O(M(n))$ operations in $\Q$, with $M(n) = 
O(n \log n \log\log n)$~\cite{DBLP:journals/acta/CantorK91}. 

Hence, for fixed $r$, Turaev-Viro invariants can be computed in 
$O(N(r) \cdot  r \log r \log\log r)$ operations in $\Q$ 
using exact arithmetic over cyclotomic fields, where $N(r)$ denotes the number 
of arithmetic operations needed to compute $\tv_{r,q}$.

\subsection{The backtracking algorithm for computing $\tv_{r,q}$}

There is a straightforward but slow algorithm to compute $\tv_{r,q}$
for arbitrary $r,q$. 
The core idea is to use a backtracking algorithm to
enumerate all admissible colourings of edges,
and compute and sum their weights.
Both major software packages that compute
Turaev-Viro invariants---the \emph{Manifold Recogniser} \cite{recogniser}
and \emph{Regina} \cite{regina}---currently employ optimised variants of this.

Let $\tri$ be a $3$-manifold 
triangulation, with $\ell$ edges $e_1, \ldots , e_\ell$.
A simple Euler characteristic argument gives $\ell = n + v$ where $n$
is the number of tetrahedra and $v$ is the number of vertices in $\tri$.
Therefore $\ell \in \Theta(n)$. 

To enumerate colourings, since each edge admits $r-1$ possible colours,
the backtracking algorithm traverses a search tree of $O((r-1)^\ell)$ nodes:
a node at depth $i$ %, $1 \leq i \leq \ell$,
corresponds to a partial colouring of the edges $e_1, \ldots , e_i$,
and each non-leaf node has $r-1$ children (one edge per colour).
Each leaf of the tree represents a 
(possibly not admissible) colouring of all the edges.
At each node we maintain the weight of the current partial colouring, and 
update this weight as we traverse the tree.
If we reach a leaf whose colouring is admissible,
we add this weight to our total.

%We observe that the backtracking search can compute $\tv_{r,q}(\tri)$
%%for a generalised triangulation $\tri$ of a closed 3-manifold
%in $O\left((r-1)^\ell \times \poly(\ell) \right)$ or 
%$O\left(C^n \allowbreak\times\allowbreak \poly(n) \right)$ 
%arithmetic operations in $\Q(\zeta)$ (with $\zeta = e^{i \pi q/r}$), 
%where $C$ is a constant for fixed $r$. More precisely, we have

\begin{lemma}
\label{lem:cpxbacktracking}
If we sort the edges $e_1, \ldots ,e_\ell$ by decreasing degree, the backtracking 
algorithm terminates in $O((r-1)^\ell)$ arithmetic operations in
$\Q(\zeta)$.
\end{lemma}

\begin{proof}
The proof is simple. The main complication is to ensure that updating
the weight of the current partial colouring takes amortised constant time.
For this we use Chebyshev's inequality, plus the observation that the
average edge degree is $\leq 6$. 

In more detail, suppose that the edges $e_1, \ldots ,e_\ell$ are ordered by decreasing degree. 
Let $\deg(e_i)$ be the degree of edge $e_i$. Changing the colour of $e_i$ affects 
the colours of the $\deg(e_i)$ triangles and $\deg(e_i)$ tetrahedra containing $e_i$. Hence 
the update of the current partial colouring weight is performed in $O(\deg(e_i))$ 
arithmetic operations in $\Q(\zeta)$. The total number of arithmetic operations 
performed by the algorithm is consequently $O(\sum_i (r-1)^i \deg(e_i))$.  Following an
Euler characteristic argument, a 
triangulation of a closed $3$-manifold with $\ell$ edges and $v$ vertices has 
$n=\ell-v$ tetrahedra and, consequently, the average degree of an edge is 
$6(\ell-v)/\ell$ and thus constant. 
Considering that the 
sequence $((r-1)^i)_i$ is increasing and $\deg e_i$ is decreasing, we conclude using 
Chebyshev's sum inequality that
$$O\left(\sum_i (r-1)^i \deg(e_i)\right) = O\left(\sum_i (r-1)^i\right) = 
O\left((r-1)^\ell\right)$$ 
\end{proof}

To obtain a bound in the number of tetrahedra $n$,
we note that a closed and connected
3-manifold triangulation with $n > 2$ tetrahedra must have
$v \leq n+1$ vertices. % \cite{burton11-asymptotic}.
Combined with $n=\ell-v$ above, we have a worst-case running time of
$O((r-1)^{2n+1})$ arithmetic operations in $\Q(\zeta)$.

\subsection{A polynomial-time algorithm for $r=3$}

Throughout this section, $\tri$ will denote an $n$-tetrahedra triangulation of 
an orientable $3$-manifold $M$.

We introduce homology with coefficients in the field $\Z_2$. 
A generalised triangulation $\tri$, after gluing, contains a set of vertices 
(dimension $0$), edges (dimension $1$), 
triangles (dimension $2$) and tetrahedra (dimension $3$) with incidence relations. 
The 
group of $d$-chains of $\tri$, $d \in \{0, \ldots ,3\}$, denoted by $\Chains_d(\tri,\Z_2)$ 
is the group of formal sums of
$d$-simplices with $\Z_2$ coefficients. The $d$th \emph{boundary operator} is a linear operator
$\partial_d: \Chains_d(\tri,\Z_2) \rightarrow \Chains_{d-1}(\tri,\Z_2)$ that assigns 
to a $d$-face the alternate 
sum of its boundary faces. The kernel of $\partial_d$, 
denoted by $\Cycles_d(\tri,\Z_2)$, is the group of \emph{$d$-cycles} and the image of $\partial_d$, 
denoted by $\Boundaries_{d-1}(\tri,\Z_2)$, is the group of \emph{$(d-1)$-boundaries}. 
The fundamental property of homology is that $\partial_d \circ \partial_{d+1}=0$, 
and we define the $d$th homology group
$\Hom_d(\tri,\Z_2)$ of $\tri$ to be the quotient $\Hom_d(\tri,\Z_2) =
\Cycles_d(\tri,\Z_2)/\Boundaries_d(\tri,\Z_2)$. It is a finite dimensional 
$\Z_2$-vector space and we denote its dimension by $\beta_d(\tri,\Z_2)$. 
For a more thorough introduction into homology theory, see~\cite{DBLP:books/daglib/0070579}.

The value of $\tv_{3,q}(\tri)$, $q \in \{1,2\}$, 
is closely related to $\Hom_2(M,\Z_2)$, the $2$-dimensional homology group
of $M$ with $\Z_2$ coefficients.
$\Hom_2(M, \Z_2)$ is a 
$\Z_2$-vector space whose dimension is the second Betti number 
$\beta_2 (M, \Z_2)$. 
Its elements are (for our purposes) equivalence classes of $2$-cycles, called \emph{homology classes}, 
which can be represented 
by $2$-dimensional triangulated surfaces $S$ embedded in $\tri$.
% With 
%$\Z_2$ coefficients, 
%$2$-cycles are collections of triangles and form a $\Z_2$ vector space. 

The \emph{Euler characteristic} of a triangulated surface $S$, denoted by $\chi(S)$, 
is $\chi(S) = v-e+f$, where $v$, $e$ and $f$ denote the 
number of vertices, edges and triangles of $S$ respectively. We define the Euler characteristic 
$\chi(c)$ of a 
$2$-cycle $c$ to be the Euler characteristic of the embedded surface it represents. 
Given $\tri$, the dimension $\beta_2 (M, \Z_2)$ of $\Hom_2(M,\Z_2)$ 
may be computed in $O(\text{poly}(n))$ operations.

The following result is well known \cite{matveev03-algms}:

\begin{proposition}[Proposition 8.1.7 in \cite{matveev03-algms}]
\label{prop:matveevr3}
Let $M$ be a closed orientable $3$-manifold. Then
$\tv_{3,2}(M)$ is equal to the order of $\Hom_2(M,\Z_2)$.
Moreover, if $M$ contains a $2$-cycle with odd Euler characteristic then
$\tv_{3,1}(M)=0$, and otherwise $\tv_{3,1}(M) = \tv_{3,2}(M)$.
\end{proposition}

Since $\Hom_2 (M, \Z_2)$ is a $\Z_2$-vector space of dimension $\beta_2(M,\Z_2)$,
we have $\tv_{3,2}(M) = 2^{\beta_2(M,\Z_2)}$, and one can compute $\tv_{3,2}(M)$ in 
polynomial time. The parity of the Euler characteristic of $2$-cycles does not 
change within a homology class; moreover, if $M$ is orientable, the map
$\Hom_2 (M, \Z_2) \to \Z_2$, taking homology classes to the parity of their 
Euler characteristic, is a homomorphism.  
Consequently, one can check whether $\tv_{3,1}(M)=0$ or $\tv_{3,1}(M) = \tv_{3,2}(M)$ 
by computing the Euler characteristic of a cycle in each of the $\beta_2(M,\Z_2)$ 
homology classes that generate $\Hom_2(M,\Z_2)$. Because $\beta_2(M,\Z_2) = O(n)$, this leads 
to a polynomial time algorithm also.

%\medskip
%In the case $\tri$ triangulates a non-orientable $3$-manifold $M$, we first 
%pass to what is called the {\em orientable double cover $\tilde{\tri}$}.
%The orientable double cover $\tilde{\tri}$ can be computed from $\tri$ in 
%polynomial time and is a $2n$-tetrahedra triangulation of an orientable closed 
%$3$-manifold $\tilde{M}$. We then apply Proposition~\ref{prop:matveevr3} to 
%$\tilde{M}$ and use the following argument to recover $\tv_{3,q}(M)$.
%
%\begin{proposition}[Matveev's book, paraphrase Thm. 4.1.1 and 4.4.1]
% get the statements
%\end{proposition}

\subsection{$\#P$-hardness of $\tv_{4,1}$}

The complexity class $\#P$ is a function class that counts accepting 
paths of a non-deterministic Turing machine~\cite{DBLP:journals/tcs/Valiant79}. 
Informally, given an \emph{NP} decision problem $C$  
asking for the existence of a solution, its $\#P$ analogue $\#C$ 
is a counting problem asking for the number of such solutions. 
A problem is \emph{$\#P$-hard} if every problem in $\#P$ 
polynomially reduces to it. For example, the problem $\#3SAT$, which asks for
the number of satisfying assignments of a $3CNF$ formula, is $\#P$-hard. %and is actually 
%$\#P$ complete. 

Naturally, counting problems are ``harder'' than their decision
counterpart,
and so $\#P$-hard problems are at least as hard as $NP$-complete
problems---specifically, $\#P$ complete problems are as hard as any 
problem in the polynomial hierarchy%, according to Toda's theorem~
\cite{DBLP:journals/siamcomp/Toda91}.
Hence proving $\#$P hardness 
is a strong complexity statement.

Kirby and Melvin \cite{kirby04-nphard} prove that computing the
Witten-Reshetikhin-Turaev invariant $\tau_r$ is $\#P$ hard for $r=4$.
This invariant $\tau_r$ is 
a more complex $3$-manifold invariant which is closely linked to the 
Turaev-Viro invariant $\tv_{r,1}$ by the formula
$\tv_{r,1}(M) = |\tau_r(M)|^2$.
Although computing $\tv_{r,1}$ is ``easier'' than computing $\tau_r$, the 
we can adapt the Kirby-Melvin hardness proof to fit our purposes. 

To prove their result, Kirby and Melvin reduce the problem of counting the zeros of a cubic 
form to the 
computation of $\tau_4$. Given a cubic form %in $\Z / 2\Z$
$$c(x_1, \ldots ,x_n) = \sum_i c_i\ x_i + \sum_{i,j} c_{ij}\  x_i x_j +
\sum_{i,j,k} c_{ijk}\ x_i x_j x_k$$
in $n$ variables over $\Z/2\Z$ and with $\#c$ zeros, they define a triangulation of a $3$-manifold 
$M_c$ with $O(\mathrm{poly}(n))$ tetrahedra satisfying
$\tau_4(M_c) = 2 \#c - 2^n$ and hence $\tv_{r,1} = (2 \#c - 2^n)^2$.

Consequently, counting the zeros of $c(x_1, \ldots ,x_n)$ reduces to computing 
$\tau_4(M_c)$, and so computing $\tv_{4,1}$ determines $\#c$ up to a 
$\pm$ sign ambiguity (depending on whether or not $c$ admits more than half of the input as zeros).

Establishing the existence of a zero for a cubic form is an \emph{NP}-complete problem,
which implies that counting the number of zeros is $\#P$ complete. Consequently, 
computing $\tau_4$ is $\#P$ hard. Kirby and Melvin prove this claim explicitly by reducing 
$\#3SAT$ to the problem of counting the zeros of a cubic form;
moreover, we observe that their construction ensures that
this cubic form admits
\emph{more than half of its inputs as zeros}.

\bigskip
We recall the reduction of $\#3SAT$ to the problem of counting the number of 
zeros of a cubic form over $\Z/2\Z$ found in~\cite{kirby04-nphard}. 
Given a $3CNF$ formula over variables $x_1, \ldots , x_n$:
$$C = C_1 \wedge \ldots \wedge C_m \ \ \ \text{with} \ \ \ 
    C_j = \bar{x}_{j_1} \vee \bar{x}_{j_2} \vee \bar{x}_{j_3}$$
and $\bar{x}_i$ is either $x_i$ or its negation $\lnot x_i$ the problem $\#3SAT$ 
consists in counting the number of assignments of "true" and "false" to 
the variables $x_1, \ldots ,x_n$ satisfying the formula.

To each form $C_j = \bar{x}_{j_1} \vee \bar{x}_{j_2} \vee \bar{x}_{j_3}$ we assign 
a cubic equation $q_j$ over $\Z/2\Z$ by setting $"\text{true}" = 0$ and 
$"\text{false}" = 1$, and 
replacing $x_i$ by the variable $x_i$ and $\lnot x_i$ by $(1-x_i)$. For example, 
a form $(\lnot x_i \wedge \lnot x_j \wedge x_k)$ leads to the equation 
$(1-x_i)(1-x_j)x_k =  0$. An assignment satisfies $C_j$ if and only if it cancels $q_j$, hence the 
number of solutions to the system of equations 
$\{ q_1(x_1, \ldots ,x_n) = 0, \ldots ,q_m(x_1,\ldots,x_n) = 0\}$ is equal to 
$\#c$, the number of satisfying assignments for $C$.

We turn each cubic equation into two quadratic equations by introducing a new variable 
$x_{ij}$ for each monomial $x_i x_j x_k$ of degree $3$ and a new quadratic equation 
$x_{ij} - x_i x_j=0$, and by replacing the product $x_i x_j x_k$ by $x_{ij} x_k$. 
We obtain a set of $m'$ equations $\{\bar{q_1}(x_1, \ldots ,x_n) = 0, \ldots ,
\bar{q_{m'}}(x_1,\ldots,x_{n'}) = 0\}$ in $n'$ variables over $\Z/2\Z$, with 
$m < m' \leq 2m$ and $n < n' \leq 2n$. 
The number of solutions of this system remains $\#c$.

Finally, we define the following cubic form $c$ by introducing $m'$ extra variables 
$z_1, \ldots ,z_{m'}$:
\[
Q = \sum_{i=1}^{m'}z_i \bar{q_i}(x_1, \ldots ,x_{n'})
\]
The number of zeros of $Q$ is equal to $2^{m'} \#c + 2^{m'-1}(2^{n'}-\#c) 
\geq \frac{1}{2} 2^{n'+m'}$. Because $Q$ is defined on $n'+m'$ variables it 
admits more than half of its input as zeros. Finally, $\#3SAT$ reduces to counting 
the number of zeros of a cubic form which admits at least half of its input as zeros.

Thus the same reduction process as for $\tau_4$ applies for $\tv_{4,1}$,
and so:

\begin{corollary}
Computing $\tv_{4,1}$ is $\#P$ hard.
\end{corollary}

\section{A fixed-parameter tractable algorithm}
\label{sec:fpt}

Here, we present an explicit fixed-parameter algorithm for computing
Turaev-Viro invariants $\tv_{r,q}$ for fixed $r$.
As is common for treewidth-based methods, the algorithm involves
dynamic programming over a tree decomposition $(X,\{B_\tau\})$.
We first describe the data that we compute and store at each bag
$B_\tau$, and then give the algorithm itself.

Our first step is to reorganise the formula for $\tv_{r,q}(\tri)$
to be a product over tetrahedra only.  This makes it easier to work with
``partial colourings'' corresponding to 
triangulation edges.

\begin{definition}
    Let $\tri$ be a generalised triangulation of a 3-manifold,
    and let $V$, $E$, $F$ and $T$ denote the vertices, edges, triangles
    and tetrahedra of $\tri$ respectively.
    For each vertex $x \in V$,
    each edge $x \in E$ and each triangle
    $x \in F$, we arbitrarily choose some tetrahedron $\Delta(x)$
    that contains $x$.

    Now consider the definition of $\tv_{r,q}(\tri)$.
    For each admissible colouring $\theta\co E \to I$
    and each tetrahedron $t \in T$, we define the
    \emph{adjusted tetrahedron weight} $|t|'_{\theta}$:
    \[ |t|'_\theta = |t|_\theta \times
        \prod_{\stackrel{v \in V}{\Delta(v)=t}} |v|_\theta \times
        \prod_{\stackrel{e \in E}{\Delta(e)=t}} |e|_\theta \times
        \prod_{\stackrel{f \in F}{\Delta(f)=t}} |f|_\theta.
    \]
    It follows from equation~(\ref{eqn-cweight}) that
    the full weight of the colouring $\theta$ is just
    \[ |\tri|_\theta = \prod_{t \in T} |t|'_\theta. \]
\end{definition}

\begin{notation}
    Let $X$ be a rooted tree.
    For any non-root node $\tau$ of $X$,
    we denote the parent node of $\tau$ by $\parent{\tau}$.
    For any two nodes $\sigma,\tau$ of $X$,
    we write $\sigma \desc \tau$ if $\sigma$ is a descendant node of $\tau$.
\end{notation}

\begin{definition}
    Let $\tri$ be a generalised triangulation of a 3-manifold,
    and let $V$, $E$, $F$ and $T$ denote the vertices, edges, triangles
    and tetrahedra of $\tri$ respectively.
    Let $(X, \{B_\tau\})$ be a nice tree decomposition of $\tri$.
    For each node $\tau$ of the rooted tree $X$, we define the following sets:
    \begin{itemize}
        \item $T_\tau \subseteq T$ is the set of all tetrahedra that appear in
        bags beneath $\tau$ but not in the bag $B_\tau$ itself.  More formally:
        $T_\tau = (\bigcup_{\sigma \desc \tau} B_\sigma ) \backslash B_\tau$.
        \item $F_\tau \subseteq F$ is the set of all triangles
        that appear in some tetrahedron $t \in T_\tau$.
        \item $E_\tau \subseteq E$ is the set of all edges
        that appear in some tetrahedron $t \in T_\tau$.
        \item $\curr_\tau \subseteq E_\tau$ is the set of all edges
        that appear in some tetrahedron $t \in T_\tau$
        and also some other tetrahedron $t' \notin T_\tau$;
        we refer to these as the \emph{current edges} at node $\tau$.
    \end{itemize}
\end{definition}

We can make the following immediate observations:

\begin{lemma}
    If $\tau$ is a leaf of the tree $X$, then we have
    $T_\tau = F_\tau = E_\tau = \curr_\tau = \emptyset$.
    If $\tau$ is the root of the tree $X$, then we have
    $T_\tau = T$, $F_\tau = F$, $E_\tau = E$, and $\curr_\tau = \emptyset$.
\end{lemma}

The key idea is, at each node $\tau$ of the tree, to store explicit
colours on the ``current'' edges $e \in \curr_\tau$ and to aggregate over all
colours on the ``finished'' edges $e \in E_\tau \backslash \curr_\tau$.
For this we need some further definitions and notation.

\begin{definition}
    Again let $\tri$ be a generalised triangulation of a 3-manifold,
    and let $(X, \{B_\tau\})$ be a nice tree decomposition of $\tri$.
    Fix some integer $r \geq 3$, and consider the set of colours
    $I = \{0, 1/2, 1, 3/2, \ldots, (r-2)/2\}$ as used in defining
    the Turaev-Viro invariants $\tv_{r,q}$.

    Let $\tau$ be any node of $X$.
    We examine ``partial colourings'' that only assign colours to the
    edges in $E_\tau$ or $\curr_\tau$:
    \begin{itemize}
        \item Consider any colouring $\theta \co E_\tau \to I$.
        We call $\theta$ \emph{admissible} if,
        for each triangle in $F_\tau$, the three edges $e,f,g$
        bounding the triangle yield an admissible triple
        $(\theta(e), \theta(f), \theta(g))$.

        \item Define $\Psi_\tau$ to be the set of all colourings
        $\psi \co \curr_\tau \to I$ that can be extended to
        \emph{any} admissible colouring $\theta \co E_\tau \to I$.

        \item Consider any colouring $\psi \in \Psi_\tau$
        (so $\psi \co \curr_\tau \to I$).
        We define the ``partial invariant''
        \[ \tv_{r,q}(\tri,\tau,\psi) =
            \sum_{\begin{subarray}{c}\theta\ \mathrm{admissible} \\
                  \theta = \psi\ \mathrm{on}\ \curr_\tau \end{subarray}}
                  \ \prod_{t \in T_\tau} |t|'_\theta. \]
    \end{itemize}
\end{definition}

Essentially, the partial invariant $\tv_{r,q}(\tri,\tau,\psi)$
considers all admissible ways $\theta$ of extending the colouring $\psi$
from the current edges $\curr_\tau$ to also include the ``finished''
edges in $E_\tau$, and then sums the partial weights
$|t|'_\theta$ for all such extensions $\theta$
using only the tetrahedra in $T_\tau$.

%Again, we can make some immediate observations:
%
%\begin{lemma}
%    If $\tau$ is a leaf of the tree $X$, then
%    since $\curr_\tau = \emptyset$ there is only one (empty) colouring
%    $\psi \co \curr_\tau \to I$, and this yields
%    $\tv_{r,q}(\tri,\tau,\psi) = 1$.
%
%    If $\tau$ is the root of the tree $X$, then again
%    since $\curr_\tau = \emptyset$ there is only one (empty) colouring
%    $\psi \co \curr_\tau \to I$, and this yields
%    $\tv_{r,q}(\tri,\tau,\psi) = \tv_{r,q}(\tri)$.
%\end{lemma}

We can now give our full fixed-parameter tractable algorithm for
$\tv_{r,q}$.

\begin{algorithm} \label{a-fpt}
    Let $\tri$ be a generalised triangulation of a 3-manifold.
    We compute $\tv_{r,q}(\tri)$ for given values of $r$ and $q$ as follows.

    Build a \emph{nice} tree decomposition $(X, \{B_\tau\})$ of $\tri$.
    Then work through each node $\tau$ of $X$ from the leaves of $X$ to
    the root, and compute $\Psi_\tau$ and $\tv_{r,q}(\tri,\tau,\psi)$
    for each $\psi \in \Psi_\tau$ as follows.
    \begin{enumerate}
        \item If $\tau$ is a leaf bag, then $\curr_\tau = E_\tau = \emptyset$,
        $\Psi_\tau$ contains just the trivial colouring $\psi$ on $\emptyset$,
        and $\tv_{r,q}(\tri,\tau,\psi) = 1$.

        \item If $\tau$ is some other introduce bag with child node $\sigma$,
        then $T_\tau = T_\sigma$.
        This means that $\Psi_\tau = \Psi_\sigma$,
        and for each $\psi \in \Psi_\tau$ we have
        $\tv_{r,q}(\tri,\tau,\psi) = \tv_{r,q}(\tri,\sigma,\psi)$.

        \item If $\tau$ is a forget bag with child node $\sigma$, then
        $T_\tau = T_\sigma \cup \{t\}$ for the unique ``forgotten''
        tetrahedron $t \in B_\tau \backslash B_\sigma$.
        Moreover, $\curr_\tau$ extends $\curr_\sigma$ by including the
        six edges of $t$ (if they were not already present).

        For each colouring $\psi \in \Psi_\sigma$, enumerate all
        possible ways of colouring the six edges of $t$ that are
        consistent with $\psi$ on any edges of $t$ that already
        appear in $\curr_\sigma$, and are admissible on the four
        triangular faces of $t$.  Each such colouring on $t$ yields an
        extension $\psi' \co \curr_\tau \to I$ of
        $\psi \co \curr_\sigma \to I$.  We include $\psi'$ in $\Psi_\tau$,
        and record the partial invariant
        $\tv_{r,q}(\tri,\tau,\psi') = \tv_{r,q}(\tri,\sigma,\psi)$.

        \item If $\tau$ is a join bag with child nodes $\sigma_1,\sigma_2$,
        then $T_\tau$ is the disjoint union $T_{\sigma_1} \dotcup T_{\sigma_2}$.
        Here $\curr_\tau$ is a subset of
        $\curr_{\sigma_1} \cup \curr_{\sigma_2}$.

        For each pair of colourings $\psi_1 \in \Psi_{\sigma_1}$ and
        $\psi_2 \in \Psi_{\sigma_2}$, if $\psi_1$ and $\psi_2$ agree on
        the common edges in $\curr_{\sigma_1} \cap \curr_{\sigma_2}$ then
        record the pair $(\psi_1,\psi_2)$.

        Each such pair yields a ``combined colouring'' in $\Psi_\tau$, which we
        denote by $\psi_1 \cdot \psi_2 \co \curr_\tau \to I$;
        note that different pairs $(\psi_1,\psi_2)$ might yield the same
        colouring $\psi_1 \cdot \psi_2$ since some
        edges from $\curr_{\sigma_1} \cup \curr_{\sigma_2}$
        might not appear in $\curr_\tau$.
        Then $\Psi_\tau$ consists of all such combined colourings
        $\psi_1 \cdot \psi_2$ from recorded pairs $(\psi_1,\psi_2)$.
        Moreover, for each combined colouring $\psi \in \Psi_\tau$ we
        compute the partial invariant $\tv_{r,q}(\tri,\tau,\psi)$
        by aggregating over all duplicates:
        \[ \tv_{r,q}(\tri,\tau,\psi) =
        \sum_{\begin{subarray}{c}
        (\psi_1,\psi_2)\ \mathrm{recorded} \\
        \psi_1 \cdot \psi_2 = \psi
        \end{subarray}}
        \tv_{r,q}(\tri,\sigma_1,\psi_1) \cdot
        \tv_{r,q}(\tri,\sigma_2,\psi_2). \]
    \end{enumerate}

    Once we have processed the entire tree,
    the root node $\rho$ of $X$ will have
    $\curr_\rho = \emptyset$,
    $\Psi_\rho$ will contain just the trivial colouring
    $\psi$ on $\emptyset$, and
    $\tv_{r,q}(\tri,\rho,\psi)$ for this trivial colouring will be
    equal to the Turaev-Viro invariant $\tv_{r,q}(\tri)$.
\end{algorithm}

%The time complexity of Algorithm~\ref{a-fpt} is simple to analyse.
%The main complication is to adjust the join operation in a way that 
%avoids \emph{squaring} the number of partial colourings.
%For details about the proof of Theorem~\ref{t-fpt}
%see Appendix~\ref{app:fpt}.

The analysis of the time complexity of this algorithm is straightforward.
Each leaf bag or introduce bag can be processed in $O(1)$ time
(of course for the introduce bag we must avoid a deep copy of the
data at the child node).
Each forget bag produces $|\Psi_\tau| \leq (r-1)^{|\curr_\tau|}$ colourings,
each of which takes $O(|\curr_\tau|)$ time to analyse.

Na{\"i}vely, each join bag requires us to process
$|\Psi_{\sigma_1}| \cdot |\Psi_{\sigma_2}| \leq
(r-1)^{|\curr_{\sigma_1}| + |\curr_{\sigma_2}|}$ pairs of
colourings $(\psi_1,\psi_2)$.  However, we can optimise this.
Since we are only interested in
colourings that agree on $\curr_{\sigma_1} \cap \curr_{\sigma_2}$,
we can first partition $\Psi_{\sigma_1}$ and $\Psi_{\sigma_2}$ into buckets
according to the colours on $\curr_{\sigma_1} \cap \curr_{\sigma_2}$,
and then combine pairs from each bucket individually.
This reduces our work to processing at most
$(r-1)^{|\curr_{\sigma_1} \cup \curr_{\sigma_2}|}$ pairs overall.
Each pair takes $O(|\curr_\tau|)$ time to process, and
the preprocessing cost for partitioning $\Psi_{\sigma_i}$ is
$O\left(|\Psi_{\sigma_i}| \cdot \log |\Psi_{\sigma_i}| \cdot
|\curr_{\sigma_i}|\right) =
O\left((r-1)^{|\curr_{\sigma_i}|} \cdot
|\curr_{\sigma_i}|^2 \log r\right)$.

Suppose that our tree decomposition has width $k$.
At each tree node $\tau$, every edge in $\curr_\tau$ must belong to some
tetrahedron in the bag $B_\tau$, and so $|\curr_\tau| \leq 6(k+1)$.
Likewise, at each join bag described above,
every edge in $\curr_{\sigma_1}$ or $\curr_{\sigma_2}$ must belong to
some tetrahedron in the bag $B_{\sigma_i}$ and therefore also
the parent bag $B_\tau$, and so
$|\curr_{\sigma_1} \cup \curr_{\sigma_2}| \leq 6(k+1)$.
From the discussion above, it follows that every bag can be processed in time
$O\left((r-1)^{6(k+1)} \cdot k^2 \log r\right)$, and so:

% TODO: This (and the rest of the section)
% needs to be reformulated in terms of native operations in
% the relevant cyclotomic field.
\begin{theorem} \label{t-fpt}
    Given a generalised triangulation $\tri$ of a 3-manifold with $n$
    tetrahedra, and a nice tree decomposition of $\tri$ with width $k$
    and $O(n)$ bags, Algorithm~\ref{a-fpt} computes $\tv_{r,q}(\tri)$
    in 
    $O\left(n \cdot (r-1)^{6(k+1)} \cdot k^2 \log r\right)$ arithmetic operations 
    in $\Q(\zeta)$.
\end{theorem}

Theorem~\ref{t-fpt} shows that, for fixed $r$, if we can keep the treewidth
small then computing $\tv_{r,q}$ becomes linear time, even
for large inputs. This of course is the main benefit of
fixed-parameter tractability.  In our setting, however, we have an
added advantage: $\tv_{r,q}$ is a topological invariant, and does
not depend on our particular choice of triangulation.

Therefore, if we are faced with a large treewidth triangulation,
we can \emph{retriangulate} the manifold (for instance, using bistellar
flips and related local moves), in an attempt to
make the treewidth smaller.  This is extremely effective in practice, as
seen in Section~\ref{s-expt}.

Even if the treewidth is large, every tree node $\tau$ satisfies
$|\curr_\tau| \leq \ell$, where $\ell$ is the number of edges in the triangulation.
Therefore the time complexity of Algorithm~\ref{a-fpt} reduces to
$O\left(n \cdot (r-1)^{\ell} \cdot \ell^2 \log r\right)$, which is only
a little slower than the backtracking algorithm (Lemma~\ref{lem:cpxbacktracking}).
This is in sharp contrast to many FPT algorithms from the literature,
which---although fast for small parameters---suffer from extremely
poor performance when the parameter becomes large.

\section{Implementation and experimentation} \label{s-expt}

Here we implement Algorithm~\ref{a-fpt} (the fixed-parameter tractable
algorithm), and subject both it and the backtracking algorithm to
exhaustive experimentation.

The FPT algorithm is implemented in the open-source software package
\emph{Regina} \cite{regina}: the source code is available from
\emph{Regina}'s public git repository, and will be included in the
next release.  For consistency we compare it to \emph{Regina}'s
long-standing implementation of the backtracking algorithm.\footnote{%
    The \emph{Manifold Recogniser} \cite{recogniser} also implements a
    backtracking algorithm, but it is not open-source and so comparisons
    are more difficult.}

In our implementation, we do not compute treewidths precisely (an
NP-complete problem)---instead, we implement the quadratic-time
\texttt{GreedyFillIn} heuristic \cite{bodlaender10-upper},
which is reported to produce small widths in practice
\cite{vandijk06-libtw}. This way, costs of building tree
decompositions are insignificant (but included in the running times).  
For both algorithms, we use
relatively na{\"i}ve implementations of arithmetic in cyclotomic
fields---these are asymptotically slower than described in
Section~\ref{s-alg}, but have very small constants.

We use two data sets for our experiments, both taken from large
``census databases'' of 3-manifolds to ensure that the
experiments are comprehensive and not cherry-picked.

The first census contains all $13\,400$ closed prime orientable
manifolds that can be formed from
$n \leq 11$ tetrahedra \cite{burton11-genus,matveev03-algms}.
This simulates ``real-world'' computation---the Turaev-Viro invariants
were used to build this census.
Since the census includes all minimal triangulations of these manifolds,
we choose the representative whose heuristic tree decomposition
has smallest width (since we are allowed to retriangulate). 

\begin{figure}[t]
    \includegraphics[scale=0.5]{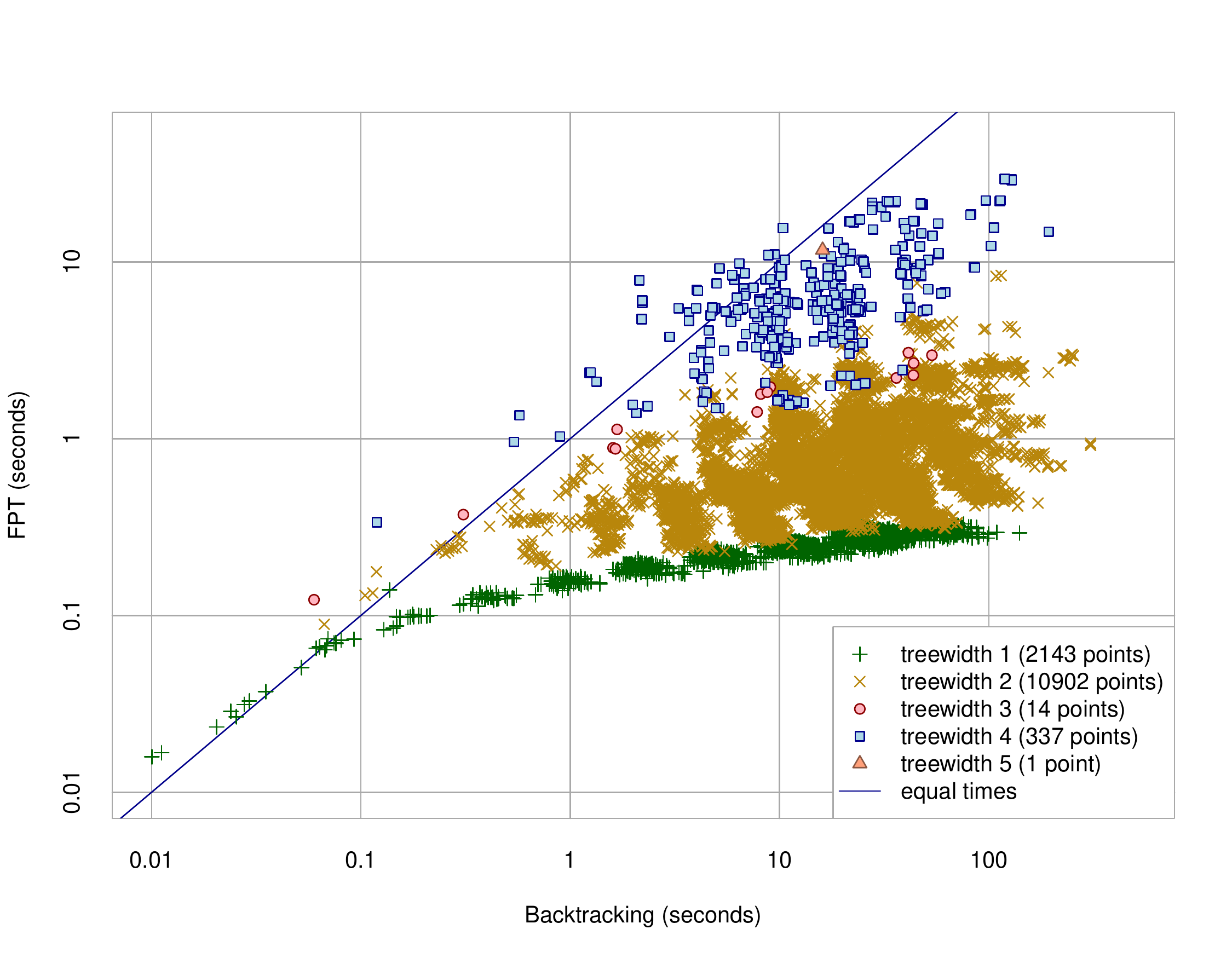}
    \caption{Running times for the 13\,400 closed census manifolds}
    \label{fig-census}
\end{figure}

The second data set contains the first $500$ (much larger) triangulations from 
the Hodgson-Weeks census of closed hyperbolic manifolds
\cite{hodgson94-closedhypcensus}.  This shows performance on larger
triangulations, with $n$ ranging from $9$ to $20$.

\begin{figure}[t]
    \includegraphics[scale=0.5]{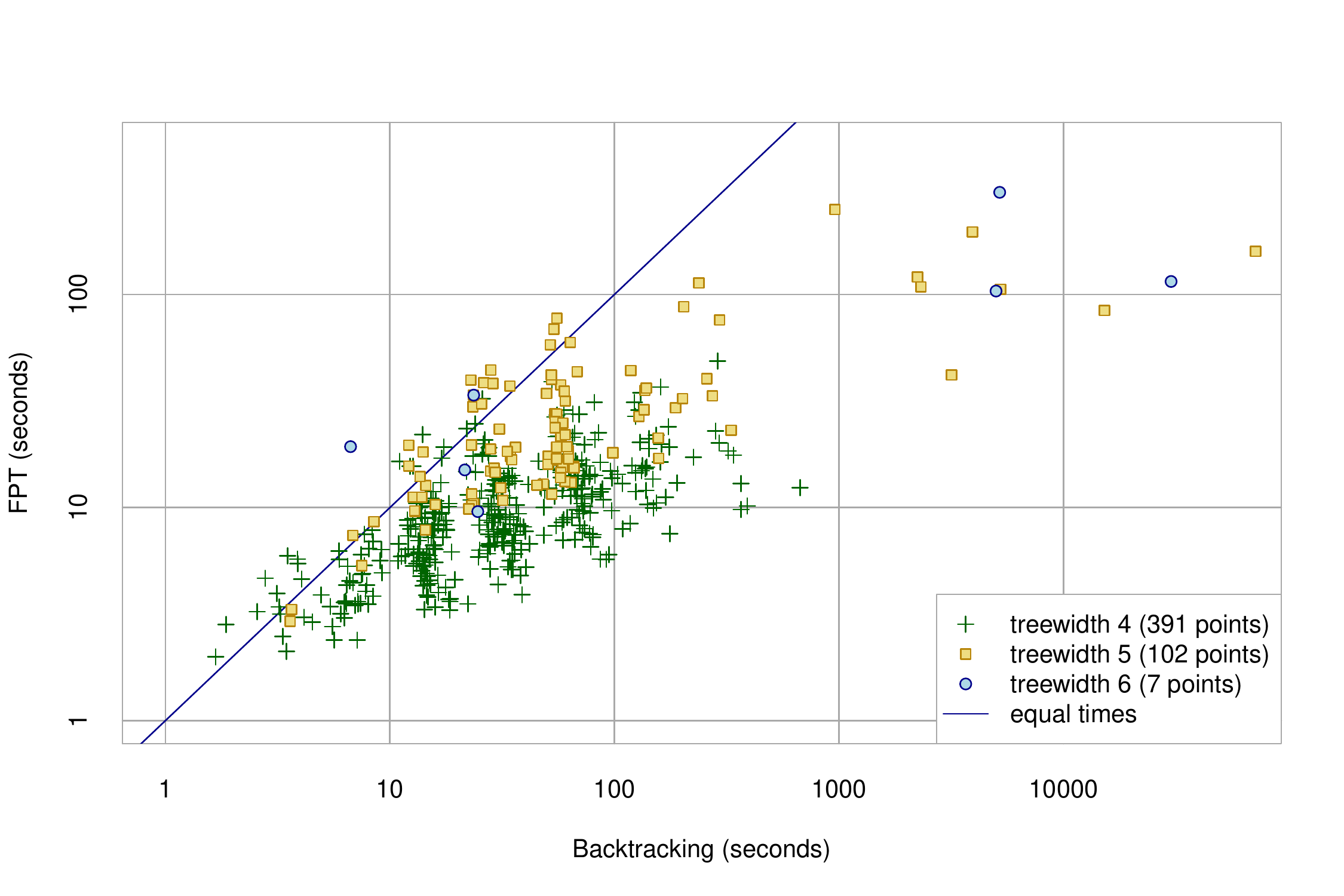}
    \caption{Running times for the first 500 Hodgson-Weeks census manifolds}
    \label{fig-hw}
\end{figure}

Figures~\ref{fig-census} and \ref{fig-hw}
compare the performance of both algorithms for each data set.  
All running times are for $\tv_{7,1}$
(the largest $r$ for which the experiments were feasible), and are
measured on a single 3\,GHz Intel Core i7 CPU.
Both plots use a log-log scale with
one data point per input triangulation.  The results are striking: the
FPT algorithm runs faster in over $99\%$ of cases, including most of the
cases with largest treewidth.
In the worst example the FPT
algorithm runs $3.7\,\times$ slower than the backtracking, but both data
sets have examples that run $>440\,\times$ faster.
It is also pleasing to see a clear impact of the treewidth
on the performance of the FPT algorithm, as one would expect.

\section{An alternate geometric interpretation}

In this section, we give a geometric interpretation of admissible 
colourings on a triangulation of a $3$-manifold $\tri$ in terms 
of {\em normal arcs}, i.e., straight lines in the interior of a triangle
which are pairwise disjoint and meet the edges of a triangle, but not its 
vertices (see Figure~\ref{fig:normalarcs}).
More precisely, we have the following

\begin{theorem}
    \label{thm:normalarcs}
    Given a $3$-manifold triangulation $\tri$, and $r \geq 3$,
    an admissible colouring of the edges of $\tri$ with $r-1$ colours 
    corresponds to a system of normal arcs
    in the 2-skeleton with $\leq r -2 $ arcs per 
    triangle forming a collection of cycles on the boundary of each tetrahedron
    of $\tri$.
\end{theorem}

\begin{proof}
  Following the definition of an admissible colouring from Section~\ref{ssec:tv},
  the colours of the edges $e_1$, $e_2$, $e_3$ of a triangle $f$ of $\tri$ must satisfy
  the parity condition, the triangle inequalities, and the upper bound constraint.

  For a colouring $\theta(e)$ of an edge $e$, we define $\phi(e) = 2\theta(e)$ which is an 
  integer; we also use the term "colouring" for $\phi$. 
  We interpret the colourings $\phi(e_1),\phi(e_2),\phi(e_3) \in \{0, 1, \ldots , r-2\}$
  as the number of intersections of normal arcs with the respective edges of
  the triangulation (see Figure~\ref{fig:normalarcs}). Without loss of 
  generality, let 
  $\phi(e_1) \geq \phi(e_2) \geq \phi(e_3)$. We construct a system of normal arcs 
  by first drawing $\phi(e_2)$ arcs between edge $e_1$ and $e_3$ and 
  $\phi(e_1) - \phi(e_2)$ arcs between edge $e_1$ and $e_3$. This is always possible
  since $\phi(e_1) \leq \phi(e_2) + \phi(e_3)$ by the triangle inequality. Furthermore,
  the parity condition ensures that an even number of unmatched 
  intersections remains which, by construction, all have to be on edge $e_3$. If this 
  number is zero we are done. Otherwise we start replacing normal arcs
  between $e_1$ and $e_2$ by pairs of normal arcs, one between $e_1$ and $e_3$ and 
  one between $e_2$ and $e_3$ (see Figure~\ref{fig:normalarcs}). In each step, the number of unmatched intersection points 
  decreases by two. By the assumption $\phi(e_2) \geq \phi(e_3)$, this 
  yields a system of normal arcs in $f$ which leaves no intersection on the
  boundary edges unmatched. This system of normal arcs
  is unique for each admissible triple of colours.
  By the upper bound constraint, we get at most $r-2$ normal arcs on $f$.

  Looking at the boundary of a tetrahedron $t$ of $\tri$ these normal arcs
  form a collection of closed cycles. To see this, note that each intersection
  point of a normal arc in a triangle with an edge is part of exactly one 
  normal arc in that triangle and
  that there are exactly two triangles sharing a given edge.
\end{proof}

\begin{figure}
  \begin{center}
    \includegraphics[width = .4 \textwidth]{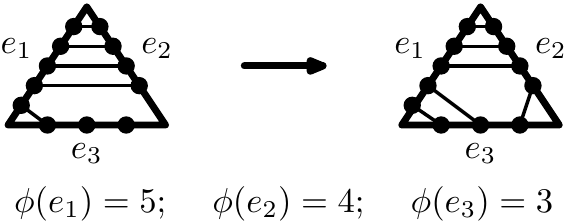}
  \end{center}
  \caption{Constructing a system of normal arcs from edge colourings. \label{fig:normalarcs}}
\end{figure}

Now, let $\tri$ be a closed $n$-tetrahedron $3$-manifold triangulation, 
$t$ a tetrahedron 
of $\tri$, $f_1$ and $f_2$ two triangles of $t$ with common edge $e$ of
colour $\phi(e)$, and $a_i$ and $b_i$ the respective non-negative numbers of the
two normal arc types in $f_i$ meeting $e$, $i \in \{1,2\}$.
Since the system of normal arcs on $t$ forms a collection of 
cycles on the boundary of $t$, we must have
%\begin{equation}
  %\label{eq:linsys}
 $a_1 + b_1 = a_2 + b_2 \leq r-2$,
%\end{equation}
giving rise to a total of $6n$ linear equations and $12n$ linear inequalities
on $6n$ variables which all admissible colourings on $\tri$ must satisfy.
Thus, finding admissible colourings on $\tri$ translates to the enumeration
of integer lattice points within the polytope defined by the above equalities
and inequalities.

Now, if we drop the upper bound constraint above,
we get a cone. Computing the Hilbert basis 
of integer lattice points of this cone yields
a finite description of all admissible colourings for any $r \geq 3$ and,
thus, the essential information to compute $\tv_{r,q} (\tri)$ for arbitrary
$r$. Transforming this approach into a practical algorithm
is work in progress.

\section{Acknowledgement}

This work is supported by the Australian Research Council
(projects DP1094516, DP140104246).

\bibliographystyle{plain}
\bibliography{pure}

\end{document}